\documentclass[12pt,leqno]{article}
\usepackage{amsfonts}
\usepackage{amsmath, amsthm, amsfonts, amssymb, color}
\usepackage{mathrsfs}
\newtheorem{thm}{Theorem}[section]

\newtheorem{lem}[thm]{Lemma}

\theoremstyle{definition}
\renewcommand{\bar}{\overline}
\renewcommand{\tilde}{\widetilde}
\newtheorem{defn}{Definition}[section]
\newcommand{\scr}[1]{\mathscr #1}
\definecolor{wco}{rgb}{0.5,0.2,0.3}

\numberwithin{equation}{section} \theoremstyle{remark}
\newtheorem{rem}{Remark}[section]

\newcommand{\ua}{\uparrow}

\title{{\bf Numerical Approximation of Stationary Distribution for SPDEs}
}
\author{
{\bf  Jianhai Bao   and   Chenggui Yuan}\\
 \footnotesize{Department of Mathematics,
Swansea University, Singleton Park, SA2 8PP, UK}\\
\footnotesize{ C.Yuan@swansea.ac.uk}}
\begin{document}
\def\R{\mathbb R}  \def\ff{\frac} \def\ss{\sqrt} \def\B{\mathbf
B}
\def\N{\mathbb N} \def\kk{\kappa} \def\m{{\bf m}}
\def\dd{\delta} \def\DD{\Dd} \def\vv{\varepsilon} \def\rr{\rho}
\def\<{\langle} \def\>{\rangle} \def\GG{\Gamma} \def\gg{\gamma}
  \def\nn{\nabla} \def\pp{\partial} \def\EE{\scr E}
\def\d{\text{\rm{d}}} \def\bb{\beta} \def\aa{\alpha} \def\D{\scr D}
  \def\si{\sigma} \def\ess{\text{\rm{ess}}}
\def\beg{\begin} \def\beq{\begin{equation}}  \def\F{\scr F}
\def\Ric{\text{\rm{Ric}}} \def\Hess{\text{\rm{Hess}}}
\def\e{\text{\rm{e}}} \def\ua{\underline a} \def\OO{\Omega}  \def\oo{\omega}
 \def\tt{\tilde} \def\Ric{\text{\rm{Ric}}}
\def\cut{\text{\rm{cut}}} \def\P{\mathbb P} \def\ifn{I_n(f^{\bigotimes n})}
\def\C{\scr C}      \def\aaa{\mathbf{r}}     \def\r{r}
\def\gap{\text{\rm{gap}}} \def\prr{\pi_{{\bf m},\varrho}}  \def\r{\mathbf r}
\def\Z{\mathbb Z} \def\vrr{\varrho} \def\l{\lambda}
\def\L{\scr L}\def\Tt{\tt} \def\TT{\tt}\def\II{\mathbb I}
\def\i{{\rm in}}\def\Sect{{\rm Sect}}\def\E{\mathbb E} \def\H{\mathbb H}
\def\M{\scr M}\def\Q{\mathbb Q} \def\texto{\text{o}} \def\LL{\Lambda}
\def\Rank{{\rm Rank}} \def\B{\scr B} \def\i{{\rm i}} \def\HR{\hat{\R}^d}
\def\to{\rightarrow}\def\l{\ell}\def\ll{\lambda}
\def\8{\infty}\def\ee{\epsilon} \def\Y{\mathbb{Y}} \def\lf{\lfloor}
\def\rf{\rfloor}\def\3{\triangle}

\maketitle

\begin{abstract}
 In this paper, we show that the exponential integrator scheme  both in spatial discretization and time
 discretization for a class of  stochastic partial differential
 equations has a unique stationary distribution
   whenever the stepsize is sufficiently small, and
 reveal that the weak limit of the law for the
exponential integrator scheme is in fact the counterpart for  the
stochastic partial differential
 equation considered.
\\

 \end{abstract}

\section{Introduction}
The convergence and the stability of numerical schemes for
finite-dimensional stochastic differential equations (SDEs) have
been extensively investigated, see, e.g., Kloeden and Platen
\cite{KP92} and Schurz \cite{S97}. Nowadays, numerical approximate
schemes for stochastic partial differential equations (SPDEs) are
also becoming more and more popular. There is extensive literature
on strong/weak convergence of approximate solutions for SPDEs. For
instance,  under a dissipative condition, Caraballo and Kloeden
\cite{ck06} showed the pathwise convergence of finite-dimensional
approximations for a class of reaction-diffusion equations. Applying
the Malliavin calculus approach, Debussche \cite{d08} discussed the
error of the Euler scheme applied to an SPDE. Greksch and Kloeden
\cite{gk96} investigated the approximation of parabolic SPDEs
through eigenfunction argument. Gy\"ongy \cite{g98}, Shardlow
\cite{s99}, and Yoo \cite{y00} applied finite differences to
approximate the mild solutions of parabolic SPDEs driven by
space-time white noise. Hausenblas \cite{h02,h03} utilized spatial
discretization  and time discretization, including implicit Euler,
explicit Euler scheme and Crank-Nicholson scheme, to approximate
quasi-linear evolution equations. Higher order pathwise numerical
approximations of SPDEs with additive noise was considered in
\cite{J11}. For the Taylor approximations of SPDEs, we refer to the
monograph \cite{JK11}.

  However, there are few results on the
 asymptotic behavior of numerical solutions for infinite-dimensional SPDEs although the counterpart for the finite-dimensional case
 has been extensively studied, see, e.g., Schurz \cite{S97}. In our present work, we shall investigate the asymptotic behavior
 of certain numerical scheme for a class of SPDEs. To begin with,
we introduce some notation and thus give the framework of our work.
Let $(H, \<\cdot,\cdot\>_H, \|\cdot\|_H)$ be a real separable
Hilbert space. Let $\mbox{id}_H: H\rightarrow H$ be the identity
operator, and denote $(\L(H),\|\cdot\|)$ and
$(\L_{HS}(H),\|\cdot\|_{HS})$ by the family of bounded linear
operators and Hilbert-Schmidt operators from $H$ into $H$,
respectively.
  In this paper, we consider an SPDE on the real separable Hilbert space
  $(H,\<\cdot,\cdot\>_H,\|\cdot\|_H)$ in the form
\begin{equation}\label{eq1}
\d X(t)=\{AX(t)+b(X(t))\}\d t+\si(X(t))\d W(t)
\end{equation}
with initial value $X(0)=x\in H$, where $W(t)$ is an $H$-valued
cylindrical $\mbox{id}_H-$Wiener process  defined on some
probability space $(\OO, \F, \P)$ with a filtration $\{\F_t\}_{t\ge0}$ satisfying the usual conditions,  $b:H\to H$ is a
Lipschitz continuous mapping, $\si(x):=\si^0+\si^1(x),x\in H$, such
that $\si^0\in\L(H)$ and $\si^1:H\to\L_{HS}(H)$.

Throughout the paper we impose the following assumptions:
\begin{enumerate}
\item[\textmd{({H1})}] $(A,\mathcal {D}(A))$  is a self-adjoint operator on $H$ generating
an immediately
 compact
 $C_0$-semigroup $\{e^{tA}\}_{t\geq0}$  such that $\|e^{tA}\|\leq e^{-\alpha
 t}$ for some $\aa>0$. In this case, by \cite[Theorem 6.26, p.185]{KT66} and \cite[Theorem 6.29, p.187]{KT66},  $-A$ has discrete spectrum
 $\{\ll_i\}_{i\ge1}$ such that
 $0<\ll_1\le\ll_2\le\cdots\le \ll_i\le\cdots$ and  $\lim_{i\to\8}\ll_i=\8$ with corresponding eigenbasis
 $\{e_i\}_{i\geq1}$ of $H$.
\item[\textmd{({H2})}] There exist $\theta_1\in(0,1)$ and $\dd_1\in(0,\8)$ such that $\int_0^t\|(-A)^{\theta_1}e^{sA}\si^0\|_{HS}^2\d
s\le\dd_1$ for any $t>0$, where
$(-A)^{\theta_1}:=\sum_{k\ge1}\ll_k^{\theta_1}(e_k\otimes e_k)$
denotes the fractional power of the operator $-A$.
\item[\textmd{({H3})}] There exist $L_1,L_2>0$
such that
\begin{equation*}
\|b(x)-b(y)\|_H\le L_1\| x-y\|_H \mbox{ and } \|
\si^1(x)-\si^1(y)\|_{HS}\le L_2\| x-y\|_H,\ \ \ x,y\in H.
\end{equation*}
%\item[\textmd{({H4})}]
%$V:=(\mathcal {D}((-A)^{\ff{1}{2}}),\|\cdot\|_1)$ is continuously,
%densely and compactly imbedded in $H$,  where
%$\|\cdot\|_1:=\|(-A)^{\ff{1}{2}}\cdot\|_H$.
\item[\textmd{({H4})}]
There exists  $\gamma\in\R$ such that
\begin{equation*}
\begin{split}
2\< x-y, b(x)-b(y)\>_H+\| \si^1(x)-\si^1(y)\|_{HS}^2\le-\gg\|
x-y\|_H^2,\ \ \ x,y\in H.
\end{split}
\end{equation*}
\end{enumerate}

By \cite[Theorem 5.3.1, p.66]{DZ96},  we know that (H1)-(H3) imply
the existence and the uniqueness of the mild solution to
\eqref{eq1}, i.e.,  there exists a unique $H$-valued adapted process
$X_x(t)$ with the initial value $x\in H$ such that
\begin{equation}\label{eq35}
X_x(t)=\e^{tA}x+\int_0^t\e^{(t-s)A}b(X_x(s))\d
s+\int_0^t\e^{(t-s)A}\si(X_x(s))\d W(s).
\end{equation}

\begin{rem}\label{Laplace}
{\rm In fact, under (H1), (H3) and $\int_0^t\|
\e^{sA}\si^0\|_{HS}^2\d s\le\dd_2$ for any $t>0$ and some $\dd_2>0$,
\eqref{eq1} also admits a unique mild solution on $H$. While (H2) is
just imposed for the later numerical analysis. Let
$\si^0=\mbox{id}_H$,  and $Ax:=\pp^2_\xi x$ for $x\in\mathcal
{D}(A):=H^2(0,\pi)\cap H^1_0(0,\pi)$. Then $A$ is a self-adjoint
negative operator and $Ae_k=-k^2e_k,\ k\in\mathbb{N}$, where
$e_k(\xi):=(2/\pi)^{1/2}\sin k\xi, \ \xi\in[0,\pi], \
k\in\mathbb{N}$. A simple computation shows that
\begin{equation*}
\int_0^t\|(-A)^{\theta_1}\e^{sA}\|_{HS}^2\d
s=\sum_{k=1}^\8(k^2)^{2\theta_1}\int_0^t\e^{-2k^2s}\d
s\le\ff{1}{2}\sum_{k=1}^\8(k^2)^{2\theta_1-1}.
\end{equation*}
Then  (H2) holds with
$\dd_1=\ff{1}{2}\sum_{k=1}^\8(k^2)^{2\theta_1-1}$ for
$\theta_1\in(0,1/4)$. }
\end{rem}

\begin{rem}
{\rm By (H3), it is readily to see that
\begin{equation}\label{eq16}
\|b(x)\|_H^2+\|\si^1(x)\|_{HS}^2\le\bar{L}(1+\|x\|_H^2),\ \ \ x\in
H,
\end{equation}
where $\bar{L}:=2((L_1^2+L_2^2)\vee\mu)$ with
$\mu:=\|b(0)\|_H^2+\|\si^1(0)\|_{HS}^2$. Moreover, by (H4) one has
\begin{equation}\label{eq29}
\begin{split}
2\<x,b(x)\>_H+\|\si^1(x)\|_{HS}^2&=2\<x,b(x)-b(0)\>_H+\|\si^1(x)-\si^1(0)\|_{HS}^2\\
&\quad+2\<x,b(0)\>_H+2\<\si^1(x)-\si^1(0),\si^1(0)\>_{HS}+\|\si^1(0)\|_{HS}^2\\
&\le-(\gg-\ee)\|x\|_H^2+2(L_2^2+1+\ee)\mu\ee^{-1},\ \ \
\ee\in(0,1),\ \ x\in H,
\end{split}
\end{equation}
where $\<T,S\>_{HS}:=\sum_{i=1}^\8\<Te_i,Se_i\>_H$ for $S,T\in
\L_{HS}(H)$.}
\end{rem}

Before establishing the numerical scheme, we further need to
introduce some notation. For any $n\in\mathbb{N}$, let $\pi_n:H\to
H_n:=\mbox{span}\{e_1,\cdots,e_n\}$ be the orthogonal projection,
i.e., $\pi_nx=\sum_{i=1}^n\langle x,e_i\rangle_He_i, x\in H$,
$A_n:=\pi_nA\in\L(H_n),b_n:=\pi_nb:H_n\to H_n$ and
$\sigma_n:=\pi_n\sigma:H_n\to\L_{HS}(H_n)$. Moreover, throughout the
paper, let $x_n:=\pi_n x$ for arbitrary $x\in U$, where $U$ is a
bounded subset of $H.$

Consider finite-dimensional  approximation associated with
\eqref{eq1} on $H_n\simeq\R^n$
\begin{equation}\label{f1}
\begin{cases}
\d X^n(t)=\{A_nX^n(t)+b_n(X^n(t))\}\d t+\sigma_n(X^n(t))\d W(t),\\
X^n(0)=x_n.
\end{cases}
\end{equation}
The spatial approximation \eqref{f1} is also called the Galerkin
approximation of \eqref{eq1}. Due to
$$\pi_nAx=\pi_nA\Big(\sum_{i=1}^n\langle
x,e_i\rangle_He_i\Big)=-\sum_{i=1}^n\langle
x,e_i\rangle_H\lambda_ie_i,\ \ \ x\in H_n,$$ it follows that
\begin{equation}\label{eqf1}
A_nx=Ax,\ \ e^{tA_n}x=e^{tA}x \ \mbox{ and }\ \langle
x,b_n(y)\rangle_H=\langle x,b(y)\rangle_H
\end{equation}
for all $x,y\in H_n.$ By (H3) and the property of the projection
operator $\pi_n$, we have
\begin{equation*}
\begin{split}
&\|A_n(x-y)+b_n(x)-b_n(y)\|_H^2+\|\si_n^1(x)-\si_n^1(y)\|_{HS}^2\\
&\le2\|A_n(x-y)\|_H^2+2\|b_n(x)-b_n(y)\|_H^2+\|\si_n^1(x)-\si_n^1(y)\|_{HS}^2\\
&\le2\ll^2_n\|x-y\|_H^2+2\|b(x)-b(y)\|_H^2+\|\si^1(x)-\si^1(y)\|_{HS}^2\\
&\le2(\ll^2_n+L_1^2+L_2^2)\|x-y\|_H^2,\ \ \ \ x,y\in H_n.
\end{split}
\end{equation*}
Hence, under (H1) and (H3), \eqref{f1} admits a unique strong
solution $\{X^n_{x_n}(t)\}_{t\ge0}$ with the starting point  $x_n\in
H_n$.

Next we introduce a time-discretization scheme for   \eqref{f1}. For
a stepsize $\3\in(0,1)$ and each integer $k\ge0$, compute the
discrete {\it Exponential Integrator} (EI) scheme
$\bar{Y}^{n,\3}_{x_n}(k\3)\approx X^n_{x_n}(k\3)$ by setting
$\bar{Y}^{n,\3}_{x_n}(0):=x_n$ and forming
\begin{equation}\label{eq7}
\bar{Y}^{n,\3}_{x_n}((k+1)\3):=\e^{\3A_n}\{\bar{Y}^{n,\3}_{x_n}(k\3)+b_n(\bar{Y}^{n,\3}_{x_n})\3+\si_n(\bar{Y}^{n,\3}_{x_n}(k\3))\3W_k\},
\end{equation}
where $\3W_k:=W((k+1)\3)-W(k\3)$,  and define the continuous EI
scheme associated with \eqref{f1} by
\begin{equation}\label{eq6}
\begin{split}
Y^{n,\3}_{x_n}(t):&=\e^{tA_n}x_n+\int_0^t\e^{(t-\lf
s\rf)A_n}b_n(Y^{n,\3}_{x_n}(\lf
s\rf))\d s\\
&\quad+\int_0^t\e^{(t-\lf s\rf)A_n}\si_n(Y^{n,\3}_{x_n}(\lf s\rf))\d W(s)\\
&=\e^{tA}x_n+\int_0^t\e^{(t-\lf s\rf)A}b_n(Y^{n,\3}_{x_n}(\lf
s\rf))\d s\\
&\quad+\int_0^t\e^{(t-\lf s\rf)A}\si_n(Y^{n,\3}_{x_n}(\lf s\rf))\d W(s)\\
\end{split}
\end{equation}
due to \eqref{eqf1}, where $\lf t\rf:=[t/\3]\3$ with  $[t/\3]$
standing for the integer part of $t/\3$.   It is easy to see from
\eqref{eq6} that
\begin{equation}\label{eq8}
\begin{split}
Y^{n,\3}_{x_n}(t)&=\e^{(t-s)A}Y^{n,\3}_{x_n}(s)+\int_s^t\e^{(t-\lf
r\rf)A}b_n(Y^{n,\3}_{x_n}(\lf r\rf))\d r\\
&\quad+\int_s^t\e^{(t-\lf r\rf)A}\si_n(Y^{n,\3}_{x_n}(\lf r\rf))\d
W(r),\ \ \ 0\le s\le t.
\end{split}
\end{equation}
By $Y^{n,\3}_{x_n}(0)=\bar{Y}^{n,\3}_{x_n}(0)$,  we deduce from
\eqref{eq7} and \eqref{eq8} that
$Y^{n,\3}_{x_n}(k\3)=\bar{Y}^{n,\3}_{x_n}(k\3)$, i.e.,
$Y^{n,\3}_{x_n}(t)$ coincides with the discrete EI approximate
solution at the gridpoints.

%\begin{rem}\label{compact}
%{\rm Under (H4), for each $r>0$ the set
%\begin{equation}\label{eq10}
%V_r:=\{x\in H:\|x\|_1\le r\}
%\end{equation}
%is a compact subset of $H$ by the definition of compact embedding.
%On the other hand,  since the  closed balls are no longer compact
%whenever $H$ is infinite-dimensional, (H4) is just imposed to
%guarantee the existence of compact subsets of $H$.}
%\end{rem}

\begin{rem}
{\rm For the finite-dimensional SDEs, the discrete Euler-Maruyama
(EM) scheme and the continuous EM scheme are standard, e.g., \cite[
p.113]{my06}. While the roots of constructing the schemes
\eqref{eq6} and \eqref{eq8} go back to,  e.g., \cite{djr11, KLNS}. }
\end{rem}

For the  discrete EI scheme \eqref{eq7}, in this paper we are
concerned with the following two questions:

\begin{itemize}
\item  Given $n\in\mathbb{N}$, for what choices of the stepsize $\3\in(0,1) $ does the EI  scheme  have a unique
stationary distribution;
\item Will the stationary distribution of the EI scheme converge weakly to  some
probability measure?  If so, what's the weak limit probability
measure?
\end{itemize}
In what follows, we shall give the positive answers to these two
questions one-by-one.

It is also worth pointing out that, for the finite-dimensional case,
Yuan and Mao \cite{ym04} studied the invariant measure of EM
numerical solutions for a class of SDEs, and Yevik and Zhao
\cite{YZ11} discussed by the global attractor approach the existence
of stationary distribution of EM scheme for SDEs which generate
random dynamical systems. Comparing the EI scheme \eqref{eq7}  with
the EM scheme  for the finite-dimensional case,  e.g., \cite[
p.113]{my06}, we note that the explicit EI schemes \eqref{eq7} is
based not only on the spatial discretization but also on the time
discretization. Moreover, in \eqref{eq1},  the linear operator $A$
is generally unbounded, and the diffusion coefficient is not
Hilbert-Schmidt, which leads to be unavailable of the It\^o formula.
Therefore, our approaches are different from those of
\cite{YZ11,ym04}. What's more, Br\'{e}hier \cite{B} investigated the
existence of invariant measure for semi-implicit Euler scheme (in
time), and discussed the numerical approximation of the invariant
measure for a class of parabolic SPDEs driven by additive noise,
where the drift coefficient is assumed to be bounded.

The organization of this paper goes as follows: In Section 2, for a
give $n\in\mathbb{N}$ and a sufficiently small stepsize
$\3\in(0,1)$, we show that the  EI approximate solution
$\{\bar{Y}^{n,\3}_{x_n}(k\3)\}_{k\ge0,x_n\in H_n}$ admits a unique
stationary distribution
  under the properties
$(\mathbb{P}1)$ and $(\mathbb{P}2)$;  Section 3 is devoting to
providing some sufficient conditions such that $(\mathbb{P}1)$ and
$(\mathbb{P}2)$ hold; In the last section, we reveal that the weak
limit of the law for the EI approximate solution
$\{\bar{Y}^{n,\3}_{x_n}(k\3)\}_{k\ge0,x_n\in H_n}$ is in fact the
counterpart for \eqref{eq1}.

\section{Stationary Distribution for the EI Scheme}
For fixed integer $n\in\mathbb{N}$,  arbitrary integer $k\ge0$ and
$\Gamma\in\mathscr{B}(H_n)$, define the $k$-step transition
probability kernel for the discrete EI approximate solution
$\bar{Y}^{n,\3}_{x_n}(k\3)$ by
\begin{equation*}
\P^{n,\3}_k(x_n,\GG):=\P(\bar{Y}^{n,\3}_{x_n}(k\3)\in\GG).
\end{equation*}

Following the argument of that of \cite[Theorem 1.2]{ym04}, we
deduce that
\begin{lem}\label{Markov Property}
{\rm $\{\bar{Y}^{n,\3}_{x_n}(k\3)\}_{k\ge0}$  is a homogeneous
Markov process. }
\end{lem}

We still need to introduce some additional notation and notions. For
a real separable Hilbert space $(K,\|\cdot\|_K)$,  let $\mathcal
{P}(K)$ stand for the collection of all probability measures on $K$.
For $P_1, P_2\in\mathcal {P}(K)$, define the metric
$\d_{\mathbb{L}}$ as follows:
\begin{equation}\label{eq9}
\d_{\mathbb{L}}(P_1,
P_2):=\sup_{f\in\mathbb{L}}\left|\int_{K}f(u)P_1(\d
u)-\int_Kf(u)P_2(\d u)\right|,
\end{equation}
where $ \mathbb{L}:=\{f:K\rightarrow
\mathbb{R}:|f(u)-f(v)|\leq\|u-v\|_K \mbox{ and }|f(\cdot)|\leq1\}.$

\begin{rem}\label{weak convergence}
It is known that the weak convergence of probability measures is a
metric concept, see, e.g., \cite[Proposition 2.5, p.6]{Iw81}. In
other words, a sequence of probability measures
$\{P_k\}_{k\geq1}\in\mathcal {P}(K)$ converges weakly to a
probability measure $P_0\in\mathcal {P}(K)$ if and only if $
\lim\limits_{k\to\8}d_{\mathbb{L}}(P_k, P_0)=0. $
\end{rem}
\begin{defn}
{\rm For a given $n\in\mathbb{N}$ and a given stepsize $\3$,
$\{\bar{Y}^{n,\3}_{x_n}(k\3)\}_{k\ge0,x_n\in H_n}$ is said to have a
stationary distribution $\pi^{n,\3}\in \mathcal {P}(H_n)$ if
$\lim\limits_{k\to\8}\d_\mathbb{L}(\P^{n,\3}_k(x_n,\cdot),\pi^{n,\3}(\cdot))=0
$ for every $x_n\in H_n.$}
\end{defn}

\begin{defn}For a given $n\in\mathbb{N}$ and a given stepsize $\3$,
$\{\bar{Y}^{n,\3}_{x_n}(k\3)\}_{k\ge0,x_n\in H_n}$ is said to have
Property $(\mathbb{P}1)$ if
\begin{equation*}
\sup_{k\ge0}\sup_{x_n\in U}\E\|\bar{Y}^{n,\3}_{x_n}(k\3)\|^2_H<\8
\end{equation*}
while it is said to have Property $(\mathbb{P}2)$ if
\begin{equation*}
\lim\limits_{k\to\8}\sup_{x_n,y_n\in
U}\E\|\bar{Y}^{n,\3}_{x_n}(k\3)-\bar{Y}^{n,\3}_{y_n}(k\3)\|_H^2=0,
\end{equation*}
where $U$ is a bounded subset of $H_n$.
\end{defn}

Our main result in this section is stated as follows.
\begin{thm}\label{numerical}
{\rm Assume that $(\mathbb{P}1)$ and $ (\mathbb{P}2)$ hold. Then,
for a given $n\in\mathbb{N}$ and a given stepsize $\3$,
$\{\bar{Y}^{n,\3}_{x_n}(k\3)\}_{k\ge0,x_n\in H_n}$ has a unique
stationary distribution $\pi^{n,\3}\in\mathcal {P}(H_n)$. }
\end{thm}

\begin{proof}
For fixed $n\in\mathbb{N}$, we note that $H_n\simeq\R^n$ is
finite-dimensional, and choose a bounded  subset $U\subseteq H_n$
such that $x_n,y_n\in U$. Following the argument to derive
\cite[Lemma 2.4 and Lemma 2.6]{ym04},  we deduce that
\begin{equation}\label{a2}
\lim_{k\to\8}\sup_{x_n,y_n\in
U}\d_\mathbb{L}(\P^{n,\3}_k(x_n,\cdot),\P^{n,\3}_k(y_n,\cdot))=0,
\end{equation}
and that, together with Lemma \ref{Markov Property}, there exists
$\pi^{n,\3}\in \mathcal {P}(H_n)$ such that
\begin{equation}\label{a3}
\lim_{k\to\8}\d_\mathbb{L}(\P^{n,\3}_k(0,\cdot),\pi^{n,\3}(\cdot))=0.
\end{equation}
Then the desired assertion follows from \eqref{a2}, \eqref{a3} and
the triangle inequality
\begin{equation*}
\d_\mathbb{L}(\P^{n,\3}_k(x_n,\cdot),\pi^{n,\3}(\cdot))\le\d_\mathbb{L}(\P^{n,\3}_k(x_n,\cdot),\P^{n,\3}_k(0,\cdot))
+\d_\mathbb{L}(\P^{n,\3}_k(0,\cdot),\pi^{n,\3}(\cdot)).
\end{equation*}
\end{proof}

\section{Sufficient Conditions for Properties $(\mathbb{P}1)$ and $(\mathbb{P}2)$}
 To make
Theorem \ref{numerical} more applicable, in this section we intend
to give some sufficient conditions such that $(\mathbb{P}1)$ and
$(\mathbb{P}2)$ hold. In what follows, $C>0$ is a generic constant
whose values may change from line to line. For notational
simplicity, let
\begin{equation*}
Z^{n,\3}(t):=\int_0^te^{(t-\lf s\rf)A}\si^0_n\d W(s) \ \mbox{ and }
\ \tilde{Y}^{n,\3}_{x_n}(t):=Y^{n,\3}_{x_n}(t)-Z^{n,\3}(t).
\end{equation*}

\begin{lem}
{\rm Under (H1)-(H3),
\begin{equation}\label{eq20}
\E\|\tilde{Y}^{n,\3}_{x_n}(t)-\tilde{Y}^{n,\3}_{x_n}(\lf
t\rf)\|_H^2\le\bb_1\3(1+\E\|\Tilde{Y}^{n,\3}_{x_n}(\lf
t\rf)\|_H^2),\ \ \ t\ge0,
\end{equation}
where
$\bb_1:=3\{(\ll_n^2+2\bar{L})\vee(2\bar{L}(1+\|(-A)^{-\theta_1}\|^2\dd_1))\}$.
 }
\end{lem}
\begin{proof}
Observe from  \eqref{eq6} that
\begin{equation}\label{eq15}
\begin{split}
\Tilde{Y}^{n,\3}_{x_n}(t)&=\e^{tA}x_n+\int_0^t\e^{(t-\lf
s\rf)A}b_n(Y^{n,\3}_{x_n}(\lf s\rf))\d s+\int_0^t\e^{(t-\lf
s\rf)A}\si^1_n(Y^{n,\3}_{x_n}(\lf s\rf))\d W(s).
\end{split}
\end{equation}
This further gives
\begin{equation*}
\begin{split}
\Tilde{Y}^{n,\3}_{x_n}(t)&=\e^{(t-\lf
t\rf)A}\Tilde{Y}^{n,\3}_{x_n}(\lf t\rf)+ \int_{\lf t\rf}^t\e^{(t-\lf
s\rf)A}b_n(Y^{n,\3}_{x_n}(\lf s\rf))\d s\\
&\quad+\int_{\lf t\rf}^t\e^{(t-\lf s\rf)A}\si^1_n(Y^{n,\3}_{x_n}(\lf
s\rf))\d W(s).
\end{split}
\end{equation*}
Then, by the  H\"older inequality, the It\^o isometry and (H1), one
has
\begin{equation}\label{eq19}
\begin{split}
&\E\|\Tilde{Y}^{n,\3}_{x_n}(t)-\Tilde{Y}^{n,\3}_{x_n}(\lf
t\rf)\|_H^2\\&\le3\Big\{\E\|(\e^{(t-\lf
t\rf)A}-\mbox{id}_H)\Tilde{Y}^{n,\3}_{x_n}(\lf t\rf)\|_H^2+
\E\int_{\lf t\rf}^t\|b(Y^{n,\3}_{x_n}(\lf s\rf))\|_H^2\d
s\\
&\quad+\E\int_{\lf t\rf}^t\|\si^1(Y^{n,\3}_{x_n}(\lf
s\rf))\|_{HS}^2\d s\Big\}\\
&=:3\{I_1(t)+I_2(t)+I_3(t)\}.
\end{split}
\end{equation}
Recalling the fundamental inequality $1-\e^{-y}\leq y, y>0,$ we
obtain from (H1) that
\begin{equation}\label{eq37}
\begin{split}
\|(\e^{(t-\lf
t\rf)A}-\mbox{id}_H)u\|^2_H&=\left\|\sum\limits_{i=1}^n(\e^{-\ll_i(t-\lf
t\rf)}-1)\< u, e_i\>_He_i\right\|_H^2\\
&\le(1-\e^{-\ll_n(t-\lf t\rf)})^2\|u\|_H^2\\
&\le\ll_n^2\3^2\|u\|_H^2,\ \ \ u\in H_n.
\end{split}
\end{equation}
Thus we arrive at
\begin{equation}\label{eq17}
I_1(t)\le\ll_n^2\3^2\E\|\Tilde{Y}^{n,\3}_{x_n}(\lf t\rf)\|_H^2.
\end{equation}
Note  from the It\^o isometry, (H1) and (H2) that
\begin{equation}\label{eq22}
\begin{split}
\E\|Z^{n,\3}(t)\|_H^2&=\int_0^t\|\e^{(s-\lf
s\rf)A}\e^{(t-s)A}\si^0_n\|_{HS}^2\d
s\le\int_0^t\|(-A)^{-\theta_1}(-A)^{\theta_1}\e^{(t-s)A}\si^0_n\|_{HS}^2\d
s\\
&\le\|(-A)^{-\theta_1}\|^2\int_0^t\|(-A)^{\theta_1}\e^{(t-s)A}\si^0\|_{HS}^2\d
s\le\|(-A)^{-\theta_1}\|^2\dd_1.
\end{split}
\end{equation}
Thus, by \eqref{eq16} and \eqref{eq22} it follows that
\begin{equation}\label{eq18}
\begin{split}
I_2(t)+I_3(t) &\le\3\E\{\|b(Y^{n,\3}_{x_n}(\lf
t\rf))\|_H^2+\|\si^1(Y^{n,\3}_{x_n}(\lf
t\rf))\|_{HS}^2\}\\
&\le2\bar{L}\3\{1+\E\|\Tilde{Y}^{n,\3}_{x_n}(\lf t\rf)\|_H^2+\E\|Z^{n,\3}(\lf t\rf)\|_H^2\}\\
&\le2\bar{L}\3\{1+\|(-A)^{-\theta_1}\|^2\dd_1+\E\|\Tilde{Y}^{n,\3}_{x_n}(\lf
t\rf)\|_H^2\}.
\end{split}
\end{equation}
As a result, \eqref{eq20} follows by substituting \eqref{eq17} and
\eqref{eq18} into \eqref{eq19}.
\end{proof}

\begin{thm}\label{boundedness}
{\rm Let (H1)-(H4) hold and assume further that $2\aa+\gg>0$. If $
\3<\min\{1,(2\aa+\gg)^2/(4\rr_1^2)\},$ then
\begin{equation}\label{eq28}
\sup_{t\ge0}\sup_{x_n\in U}\E\|Y^{n,\3}_{x_n}(t)\|_H^2<\8,
\end{equation}
where
$\rr_1:=2+(|14\alpha-\gg|^2/64+2\bar{L}+|14\alpha-\gg|/8)\bb_1+2(1+\bb_1+\ll_n^2\bar{L})$ and $U$ is a bounded subset of $H_n$.
Hence Property $(\mathbb{P}1)$ holds whenever the stepsize $\3$ is
sufficiently small.  }
\end{thm}
\begin{proof}
 Note that
\eqref{eq15} can be rewritten in the differential form
\begin{equation}\label{eq21}
\begin{split}
\d
\Tilde{Y}^{n,\3}_{x_n}(t)&=\{A\Tilde{Y}^{n,\3}_{x_n}(t)+\e^{(t-\lf
t\rf)A}b_n(Y^{n,\3}_{x_n}(\lf t\rf))\}\d t+ \e^{(t-\lf
t\rf)A}\si^1_n(Y^{n,\3}_{x_n}(\lf t\rf))\d W(t)
\end{split}
\end{equation}
with $\Tilde{Y}^{n,\3}_{x_n}(0)=x_n$.  For any $\nu>0$, by the It\^o
formula we derive from \eqref{eq21} and (H1) that
\begin{equation}\label{eq14}
\begin{split}
&\E(\e^{\nu
t}\|\Tilde{Y}^{n,\3}_{x_n}(t)\|_H^2)\\&\le\|x\|_H^2+\E\int_0^t\e^{\nu
s}\{\nu\|\Tilde{Y}^{n,\3}_{x_n}(s)\|_H^2 +2\<\Tilde{Y}^{n,\3}_{x_n}(s),A\Tilde{Y}^{n,\3}_{x_n}(s)\>_H\\
&\quad+2\<\Tilde{Y}^{n,\3}_{x_n}(s),\e^{(s-\lf s\rf)A}b_n(Y^{n,\3}_{x_n}(\lf s\rf))\>_H+\|\e^{(s-\lf s\rf)A}\si^1_n(Y^{n,\3}_{x_n}(\lf s\rf))\|_{HS}^2\}\d s\\
&\leq\|x\|_H^2+\E\int_0^t\e^{\nu
s}\{-(2\alpha-\nu)\|\Tilde{Y}^{n,\3}_{x_n}(s)\|_H^2 \\
&\quad+2\<\Tilde{Y}^{n,\3}_{x_n}(s),\e^{(s-\lf
s\rf)A}b_n(Y^{n,\3}_{x_n}(\lf s\rf))\>_H+\|\si^1(Y^{n,\3}_{x_n}(\lf
s\rf))\|_{HS}^2\}\d s.
\end{split}
\end{equation}
Since
\begin{equation}\label{eq24}
\begin{split}
\|\Tilde{Y}^{n,\3}_{x_n}(t)\|_H^2&=\|\Tilde{Y}^{n,\3}_{x_n}(\lf
t\rf)\|_H^2+2\<\Tilde{Y}^{n,\3}_{x_n}(\lf
t\rf),\Tilde{Y}^{n,\3}_{x_n}(t)-\Tilde{Y}^{n,\3}_{x_n}(\lf
t\rf)\>_H\\
&\quad+\|\Tilde{Y}^{n,\3}_{x_n}(t)-\Tilde{Y}^{n,\3}_{x_n}(\lf
t\rf)\|_H^2,
\end{split}
\end{equation}
and
\begin{equation*}
\begin{split}
&\<\Tilde{Y}^{n,\3}_{x_n}(t),\e^{(t-\lf t\rf)A}b_n(Y^{n,\3}_{x_n}(\lf t\rf))\>_H\\&=\<Y^{n,\3}_{x_n}(\lf t\rf),b(Y^{n,\3}_{x_n}(\lf t\rf))\>_H+\<\Tilde{Y}^{n,\3}_{x_n}(t)-\Tilde{Y}^{n,\3}_{x_n}(\lf t\rf),b(Y^{n,\3}_{x_n}(\lf t\rf))\>_H\\
&\quad-\<Z^{n,\3}(\lf t\rf),b(Y^{n,\3}_{x_n}(\lf
t\rf))\>_H+\<\Tilde{Y}^{n,\3}_{x_n}(t),(\e^{(t-\lf
t\rf)A}-\mbox{id}_H)b_n(Y^{n,\3}_{x_n}(\lf t\rf))\>_H,
\end{split}
\end{equation*}
it follows from  \eqref{eq14} that
\begin{equation*}
\begin{split}
\E(\e^{\nu
t}\|\Tilde{Y}^{n,\3}_{x_n}(t)\|_H^2)&\leq\|x\|_H^2+\E\int_0^t\e^{\nu
s}\{-(2\alpha-\nu)\|\Tilde{Y}^{n,\3}_{x_n}(\lf s\rf)\|_H^2+\|\si^1(Y^{n,\3}_{x_n}(\lf s\rf))\|_{HS}^2\\
&\quad+2\<Y^{n,\3}_{x_n}(\lf s\rf),b(Y^{n,\3}_{x_n}(\lf s\rf))\>_H\\
&\quad-2(2\alpha-\nu)\<\Tilde{Y}^{n,\3}_{x_n}(\lf s\rf),\Tilde{Y}^{n,\3}_{x_n}(s)-\Tilde{Y}^{n,\3}_{x_n}(\lf s\rf)\>_H\\
&\quad-(2\alpha-\nu)\|\Tilde{Y}^{n,\3}_{x_n}(s)-\Tilde{Y}^{n,\3}_{x_n}(\lf s\rf)\|_H^2 \\
&\quad+2\<\Tilde{Y}^{n,\3}_{x_n}(s)-\Tilde{Y}^{n,\3}_{x_n}(\lf s\rf),b(Y^{n,\3}_{x_n}(\lf s\rf))\>_H\\
&\quad-2\<Z^{n,\3}(\lf s\rf)),b(Y^{n,\3}_{x_n}(\lf
s\rf))\>_H\\
&\quad+2\<\Tilde{Y}^{n,\3}_{x_n}(s),(\e^{(s-\lf
s\rf)A}-\mbox{id}_H)b_n(Y^{n,\3}_{x_n}(\lf s\rf))\>_H\}\d s.
\end{split}
\end{equation*}
This, together with \eqref{eq29}, yields that
\begin{equation}\label{eq27}
\begin{split}
\E(\e^{\nu
t}\|\Tilde{Y}^{n,\3}_{x_n}(t)\|_H^2)&\leq\|x\|_H^2-(2\alpha+\gg-\ee-\nu)\E\int_0^t\e^{\nu
s}\|\Tilde{Y}^{n,\3}_{x_n}(\lf s\rf)\|_H^2\d s\\
&\quad+\E\int_0^t\e^{\nu
s}\{-2(2\alpha-\nu)\<\Tilde{Y}^{n,\3}_{x_n}(\lf s\rf),\Tilde{Y}^{n,\3}_{x_n}(s)-\Tilde{Y}^{n,\3}_{x_n}(\lf s\rf)\>_H\\
&\quad-(2\alpha-\nu)\|\Tilde{Y}^{n,\3}_{x_n}(s)-\Tilde{Y}^{n,\3}_{x_n}(\lf s\rf)\|_H^2 \\
&\quad+2\<\Tilde{Y}^{n,\3}_{x_n}(s)-\Tilde{Y}^{n,\3}_{x_n}(\lf s\rf),b(Y^{n,\3}_{x_n}(\lf s\rf))\>_H\}\d s\\
&\quad+2\E\int_0^t\e^{\nu s}\<\Tilde{Y}^{n,\3}_{x_n}(s),(\e^{(s-\lf
s\rf)A}-\mbox{id}_H)b_n(Y^{n,\3}_{x_n}(\lf s\rf))\>_H\d s\\
&\quad+\E\int_0^t\e^{\nu
s}\{2(L_2^2+1+\ee^{-1})\mu\ee^{-1}-2\<Z^{n,\3}(\lf
s\rf),b(Y^{n,\3}_{x_n}(\lf
s\rf))\>_H\\
&\quad-2(\gg-\ee)\<Z^{n,\3}(\lf s\rf),\Tilde{Y}^{n,\3}_{x_n}(\lf
s\rf)\>_H-(\gg-\ee)\|Z^{n,\3}(\lf
s\rf)\|_H^2\}\d s\\
&=:J_1(t)+J_2(t)+J_3(t)+J_4(t).
\end{split}
\end{equation}
By the elemental inequality: $2ab\le\kk a^2+b^2/\kk,a,b\in\R,\kk>0$,
and \eqref{eq20}, we arrive at
\begin{equation*}
\begin{split}
J_2(t)
%\le\E\int_0^te^{\nu
%s}\{2|2\alpha-\nu|\cdot\|\Tilde{Y}^{n,\3}_{x_n}(\lf s\rf)\|_H\|\Tilde{Y}^{n,\3}_{x_n}(s)-\Tilde{Y}^{n,\3}_{x_n}(\lf s\rf)\|_H\\
%&\quad+|2\alpha-\nu|\cdot\|\Tilde{Y}^{n,\3}_{x_n}(s)-\Tilde{Y}^{n,\3}_{x_n}(\lf s\rf)\|_H^2 +2\|\Tilde{Y}^{n,\3}_{x_n}(s)-\Tilde{Y}^{n,\3}_{x_n}(\lf s\rf)\|_H\|b(Y^{n,\3}_{x_n}(\lf s\rf))\|_H\}\d s\\
&\le\E\int_0^t\e^{\nu
s}\Big\{\3^{\ff{1}{2}}\|\Tilde{Y}^{n,\3}_{x_n}(\lf
s\rf)\|_H^2+2^{-1}\bar{L}^{-1}\3^{\ff{1}{2}}\|b(Y^{n,\3}_{x_n}(\lf
s\rf))\|_H^2\\
&\quad+\{(|2\alpha-\nu|^2+2\bar{L})\3^{-\ff{1}{2}}+|2\alpha-\nu|\}\|\Tilde{Y}^{n,\3}_{x_n}(s)-\Tilde{Y}^{n,\3}_{x_n}(\lf
s\rf)\|_H^2\Big\}\d s\\
&\le\E\int_0^t\e^{\nu
s}\Big\{2\3^{\ff{1}{2}}\|\Tilde{Y}^{n,\3}_{x_n}(\lf
s\rf)\|_H^2+2^{-1}\3^{\ff{1}{2}}+\3^{\ff{1}{2}}\|Z^{n,\3}(\lf
s\rf)\|_H^2\\
&\quad+\{(|2\alpha-\nu|^2+2\bar{L})\3^{-\ff{1}{2}}+|2\alpha-\nu|\}\|\Tilde{Y}^{n,\3}_{x_n}(s)-\Tilde{Y}^{n,\3}_{x_n}(\lf
s\rf)\|_H^2\Big\}\d s,
\end{split}
\end{equation*}
where in the last step we have used \eqref{eq16}. Combining
\eqref{eq20} with \eqref{eq22}, we thus obtain that
\begin{equation}\label{eq23}
\begin{split}
J_2(t) &\le\int_0^t\e^{\nu
s}\Big\{\{2+(|2\alpha-\nu|^2+2\bar{L}+|2\alpha-\nu|)\bb_1\}\3^{\ff{1}{2}}\E\|\Tilde{Y}^{n,\3}_{x_n}(\lf
t\rf)\|_H^2\\
&\quad+\{1+\|(-A)^{-\theta_1}\|^2\dd_1+(|2\alpha-\nu|^2+2\bar{L}+|2\alpha-\nu|)\bb_1\}\3^{\ff{1}{2}}\Big\}\d
s.
\end{split}
\end{equation}
On the other hand, we deduce from \eqref{eq16}, \eqref{eq37},
\eqref{eq22} and \eqref{eq24}
 that
\begin{equation}\label{eq25}
\begin{split}
J_3(t)&\le\E\int_0^t\e^{\nu
s}\{\3^{\ff{1}{2}}\|\Tilde{Y}^{n,\3}_{x_n}(\lf
t\rf)\|_H^2+2\3^{\ff{1}{2}}\<\Tilde{Y}^{n,\3}_{x_n}(\lf
t\rf),\Tilde{Y}^{n,\3}_{x_n}(t)-\Tilde{Y}^{n,\3}_{x_n}(\lf
t\rf)\>_H\\
&\quad+\3^{\ff{1}{2}}\|\Tilde{Y}^{n,\3}_{x_n}(t)-\Tilde{Y}^{n,\3}_{x_n}(\lf
t\rf)\|_H^2+\3^{-\ff{1}{2}}\|(\e^{(s-\lf
s\rf)A}-\mbox{id}_H)b_n(Y^{n,\3}_{x_n}(\lf s\rf))\|_H^2\}\d s\\
&\le\E\int_0^t\e^{\nu
s}\{2\3^{\ff{1}{2}}\|\Tilde{Y}^{n,\3}_{x_n}(\lf
t\rf)\|_H^2+2\3^{\ff{1}{2}}\|\Tilde{Y}^{n,\3}_{x_n}(t)-\Tilde{Y}^{n,\3}_{x_n}(\lf
t\rf)\|^2_H\\
&\quad+\3^{-\ff{1}{2}}\|(\e^{(s-\lf
s\rf)A}-\mbox{id}_H)b_n(Y^{n,\3}_{x_n}(\lf s\rf))\|_H^2\}\d s\\
%&\le\int_0^te^{\nu s}\{2\3^{\ff{1}{2}}\E\|Z(\lf
%t\rf)\|_H^2+2\bb\3^{\ff{1}{2}}(1+\E\|Z(\lf t\rf)\|_H^2)\\
%&\quad+\ll_1^2\3^{\ff{1}{2}}\bar{L}(1+\dd+2\E\|Z(\lf s\rf)\|_H^2\}\d s\\
&\le\int_0^t\e^{\nu
s}\{2(1+\bb_1+\ll_n^2\bar{L})\3^{\ff{1}{2}}\E\|\Tilde{Y}^{n,\3}_{x_n}(\lf
t\rf)\|_H^2\\
&\quad+2(\bb_1+\ll_n^2\bar{L}(1+\|(-A)^{-\theta_1}\|^2\dd_1))\3^{\ff{1}{2}}\}\d
s.
\end{split}
\end{equation}
Furthermore, due to \eqref{eq16} and \eqref{eq22}, for arbitrary
$\kk>0$ one has
\begin{equation*}
\begin{split}
J_4(t)&\le\E\int_0^t\e^{\nu
s}\{2(L_2^2+1+\ee^{-1})\mu\ee^{-1}+2\|Z^{n,\3}(\lf
s\rf)\|_H\|b(Y^{n,\3}_{x_n}(\lf
s\rf))\|_H\\
&\quad+2|\gg-\ee|\cdot\|Z^{n,\3}(\lf
s\rf)\|_H\|\Tilde{Y}^{n,\3}_{x_n}(\lf
s\rf)\|_H+|\gg-\ee|\cdot\|Z^{n,\3}(\lf s\rf)\|_H^2\}\d s\\
&\le\E\int_0^t\e^{\nu
s}\{2(L_2^2+1+\ee^{-1})\mu\ee^{-1}+\kk^{-1}\|Z^{n,\3}(\lf
s\rf)\|_H^2+\kk\|b(Y^{n,\3}_{x_n}(\lf
s\rf))\|_H^2\\
&\quad+|\gg-\ee|^2\kk^{-1}\|Z^{n,\3}(\lf
s\rf)\|_H+\kk\|\Tilde{Y}^{n,\3}_{x_n}(\lf
s\rf)\|_H^2+|\gg-\ee|\cdot\|Z^{n,\3}(\lf s\rf)\|_H^2\}\d s\\
%&\le\E\int_0^te^{\nu s}\{\ee^{-1}\|\tt{Y}(\lf
%s\rf)\|_H^2+\ee\bar{L}(1+2\|Z(\lf
%s\rf)\|_H^2+2\|\tt{Y}(\lf s\rf)\|_H^2)\\
%&\quad+|\gamma|^2\ee^{-1}\|\tt{Y}(\lf s\rf)\|_H+\ee\|Z(\lf
%s\rf)\|_H^2+|\gg|\|\tt{Y}(\lf s\rf)\|_H^2\}\d s\\
&\le\int_0^t\e^{\nu
s}\{\kk\bar{L}+(\kk^{-1}+2\kk\bar{L}+|\gg-\ee|^2\kk^{-1}+|\gg-\ee|)\|(-A)^{-\theta_1}\|^2\dd_1\\
&\quad+2(L_2^2+1+\ee^{-1})\mu\ee^{-1}+(1+2\bar{L})\kk\E\|\Tilde{Y}^{n,\3}_{x_n}(\lf
s\rf)\|_H^2\}\d s.
\end{split}
\end{equation*}
In particular, taking $\ee=\nu=(2\alpha+\gg)/8$ and
$\kk=(2\alpha+\gg)/(4(1+2\bar{L}))$ yields that
\begin{equation}\label{eq26}
\begin{split}
J_4(t) &\le\int_0^t\e^{\nu
s}\{4^{-1}(2\alpha+\gg)\E\|\Tilde{Y}^{n,\3}_{x_n}(\lf
s\rf)\|_H^2+C\}\d s.
\end{split}
\end{equation}
 Putting
\eqref{eq23}-\eqref{eq26} into \eqref{eq27}, we deduce that
\begin{equation}\label{eq32}
\begin{split}
\E(\e^{\nu
t}\|\Tilde{Y}^{n,\3}_{x_n}(t)\|_H^2)&\leq\|x\|_H^2+C\int_0^t\e^{\nu
s}\d s\\
&\quad-\ff{2\alpha+\gg-2\rr_1\3^{\ff{1}{2}}}{2}\E\int_0^t\e^{\nu
s}\|\Tilde{Y}^{n,\3}_{x_n}(\lf s\rf)\|_H^2\d s.
\end{split}
\end{equation}
For $\3<(2\aa+\gg)^2/(4\rr^2_1),$ it is trivial to see that
$2\alpha+\gg-2\rr_1\3^{\ff{1}{2}}>0.$ Thus we have
\begin{equation*}
\sup_{t\ge0}\sup_{x_n\in
U}\E(\|\Tilde{Y}^{n,\3}_{x_n}(t)\|_H^2)<\8.
\end{equation*} Finally,
\eqref{eq28} follows by recalling
$\Tilde{Y}^{n,\3}_{x_n}(t)=Y^{n,\3}_{x_n}(t)-Z^{n,\3}(t)$ and
\eqref{eq22}.
\end{proof}

\begin{thm}\label{stability}
{\rm Let the assumptions of Lemma \ref{boundedness} hold. If
$\3<\min\{1,(2\aa+\gg)^2/(4\rr_2^2)\}$, then
\begin{equation}\label{eq33}
\lim\limits_{t\to\8}\sup_{x_n,y_n\in
U}\E\|Y^{n,\3}_{x_n}(t)-Y^{n,\3}_{y_n}(t)\|_H^2=0,
\end{equation}
where $
\rr_2:=6(\ll_n^2+\bar{L})(|2\aa-\gg|+1)+3+7\bar{L}+\ll_n^2\bar{L}+6\ll_n^2$ and $U$ is a bounded subset of $H_n$.
Hence Property $(\mathbb{P}2)$ holds whenever the stepsize $\3$ is
sufficiently small.
 }
\end{thm}

\begin{proof}
Let
\begin{equation*}
Z^{n,\3}_{x_n,y_n}(t):=Y^{n,\3}_{x_n}(t)-Y^{n,\3}_{y_n}(t).
\end{equation*}
Note from \eqref{eq6}  that
\begin{equation*}
\begin{split}
Z^{n,\3}_{x_n,y_n}(t)-Z^{n,\3}_{x_n,y_n}(\lf t\rf)&=(\e^{(t-\lf
t\rf)A}-\mbox{id}_H)Z^{n,\3}_{x_n,y_n}(\lf
t\rf)\\
&\quad+ \int_{\lf t\rf}^t\e^{(t-\lf s\rf)A}(b_n(Y^{n,\3}_{x_n}(\lf
s\rf))-b_n(Y^{n,\3}_{y_n}(\lf s\rf)))\d s\\
&\quad+\int_{\lf t\rf}^t\e^{(t-\lf
s\rf)A}(\si^1_n(Y^{n,\3}_{x_n}(\lf s\rf))-\si^1(Y^{n,\3}_{y_n}(\lf
s\rf)))\d W(s).
\end{split}
\end{equation*}
Following the argument of that of \eqref{eq20}, we derive that
\begin{equation}\label{eq31}
\begin{split}
\E\|Z^{n,\3}_{x_n,y_n}(t)-Z^{n,\3}_{x_n,y_n}(\lf
t\rf)\|_H^2\leq3(\ll_n^2+\bar{L})\3\E\|Z^{n,\3}_{x_n,y_n}(\lf
t\rf)\|_H^2.
\end{split}
\end{equation}
For $\nu:=(2\aa+\gg)/2$, by the It\^o formula it follows from
\eqref{eq6}, (H1) and (H4) that
\begin{equation*}
\begin{split}
\E(\e^{\nu t}\|Z^{n,\3}_{x_n,y_n}(t)\|_H^2)
%&=\|x-y\|_H^2+\nu\E\int_0^te^{\nu
%s}\|Z^{x,y}(s)\|_H^2\d s\\
%&\quad+\E\int_0^te^{\nu
%s}\{2\<Z^{x,y}(s),AZ^{x,y}(s)\>_H\\
%&\quad+2\<Z^{x,y}(s),e^{(s-\lf s\rf)A}(b(Y^x(\lf s\rf))-b(Y^y(\lf
%s\rf)))\>_H\\
%&\quad+\|e^{(t-\lf t\rf)A}(\si^1(Y^x(\lf t\rf))-\si^1(Y^y(\lf
%t\rf)))\|_{HS}^2\}\d s\\
&\le\|x-y\|_H^2+\nu\E\int_0^t\e^{\nu
s}\|Z^{n,\3}_{x_n,y_n}(s)\|_H^2\d s\\
&\quad+\E\int_0^t\e^{\nu
s}\{2\<Z^{n,\3}_{x_n,y_n}(s),AZ^{n,\3}_{x_n,y_n}(s)\>_H\\
&\quad+2\<Z^{n,\3}_{x_n,y_n}(\lf s\rf),b(Y^{n,\3}_{x_n}(\lf
s\rf))-b(Y^{n,\3}_{y_n}(\lf
s\rf))\>_H\\
&\quad+\|\si^1(Y^{n,\3}_{x_n}(\lf s\rf))-\si^1(Y^{n,\3}_{y_n}(\lf s\rf))\|_{HS}^2\\
&\quad+2\<Z^{n,\3}_{x_n,y_n}(s)-Z^{n,\3}_{x_n,y_n}(\lf
s\rf),b(Y^{n,\3}_{x_n}(\lf s\rf))-b(Y^{n,\3}_{y_n}(\lf
s\rf))\>_H\\
&\quad+2\<Z^{n,\3}_{x_n,y_n}(s),(e^{(s-\lf
s\rf)A}-\mbox{id}_H)(b_n(Y^{n,\3}_{x_n}(\lf
s\rf))-b_n(Y^{n,\3}_{y_n}(\lf s\rf)))\>_H\}\d s\\
&\le\|x-y\|_H^2-(2\aa+\gg-\nu)\E\int_0^t\e^{\nu
s}\|Z^{n,\3}_{x_n,y_n}(\lf s\rf)\|_H^2\d s\\
&\quad+\E\int_0^t\e^{\nu s}\{-2(2\aa-\nu)\<Z^{n,\3}_{x_n,y_n}(\lf s\rf),Z^{n,\3}_{x_n,y_n}(s)-Z^{n,\3}_{x_n,y_n}(\lf s\rf)\>_H\\
&\quad-(2\aa-\nu)\|Z^{n,\3}_{x_n,y_n}(s)-Z^{n,\3}_{x_n,y_n}(\lf s\rf)\|_H^2\\
 &\quad+2\<Z^{n,\3}_{x_n,y_n}(s)-Z^{n,\3}_{x_n,y_n}(\lf s\rf),b(Y^{n,\3}_{x_n}(\lf
s\rf))-b(Y^{n,\3}_{y_n}(\lf
s\rf))\>_H\}\d s\\
&\quad+2\E\int_0^te^{\nu s}\<Z^{n,\3}_{x_n,y_n}(s),(e^{(s-\lf
s\rf)A}-\mbox{id}_H)(b_n(Y^{n,\3}_{x_n}(\lf s\rf))-b_n(Y^{n,\3}_{y_n}(\lf s\rf)))\>_H\d s\\
&=:\bar{J}_1(t)+\bar{J}_2(t)+\bar{J}_3(t).
\end{split}
\end{equation*}
where we have also used the \eqref{eq24} with
$\tilde{Y}^{n,\3}_{x_n}(t)$ replaced by $Z^{n,\3}_{x_n,y_n}(t)$. By
(H3) and \eqref{eq31}, one has
\begin{equation*}
\begin{split}
\bar{J}_2(t)&\le\E\int_0^t\e^{\nu
s}\{\3^{\ff{1}{2}}\|Z^{n,\3}_{x_n,y_n}(\lf
s\rf)\|_H^2+\3^{\ff{1}{2}}\|b(Y^{n,\3}_{x_n}(\lf
s\rf))-b(Y^{n,\3}_{y_n}(\lf
s\rf))\|_H^2\\
&\quad+\{|2\aa-\nu|+(|2\aa-\nu|+1)\3^{-\ff{1}{2}}\}\|Z^{n,\3}_{x_n,y_n}(s)-Z^{n,\3}_{x_n,y_n}(\lf s\rf)\|_H^2\}\d s\\
&\le\{6(\ll_n^2+\bar{L})(|2\aa-\nu|+1)+1+\bar{L}\}\3^{\ff{1}{2}}\E\int_0^t\e^{\nu
s}\|Z^{n,\3}_{x_n,y_n}(\lf s\rf)\|_H^2\d s.
\end{split}
\end{equation*}
On the other hand, carrying out a similar argument to that of
\eqref{eq25} leads to
\begin{equation*}
\begin{split}
\bar{J}_3(t) &\le2\E\int_0^t\e^{\nu
s}\{\3^{\ff{1}{2}}\|Z^{n,\3}_{x_n,y_n}(s)-Z^{n,\3}_{x_n,y_n}(\lf s\rf)\|_H^2+\3^{\ff{1}{2}}\<Z^{n,\3}_{x_n,y_n}(s)-Z^{x,y}(\lf s\rf),Z^{n,\3}_{x_n,y_n}(\lf s\rf)\>_H\\
&\quad+\3^{\ff{1}{2}}\|Z^{n,\3}_{x_n,y_n}(\lf
s\rf)\|_H^2+\3^{-\ff{1}{2}}\|(\e^{(s-\lf
s\rf)A}-\mbox{id}_H)(b_n(Y^{n,\3}_{x_n}(\lf
s\rf))-b_n(Y^{n,\3}_{y_n}(\lf s\rf)))\|_H^2\d s\\
&\le(2+\ll_n^2\bar{L}+6\ll_n^2+6\bar{L})\3^{\ff{1}{2}}\E\int_0^t\e^{\nu
s}\|Z^{n,\3}_{x_n,y_n}(\lf s\rf)\|_H^2\d s.
\end{split}
\end{equation*}
Hence we arrive at
\begin{equation*}
\begin{split}
\E(\e^{\nu
t}\|Z^{n,\3}_{x_n,y_n}(t)\|_H^2)\le\|x-y\|_H^2-\ff{2\aa+\gg-2\rr_2\3^{\ff{1}{2}}}{2}\E\int_0^t\e^{\nu
s}\|Z^{n,\3}_{x_n,y_n}(\lf s\rf)\|_H^2\d s,
\end{split}
\end{equation*}
and then the desired assertion \eqref{eq33} follows by
$\3\le\min\{1,(2\aa+\gg)^2/(4\rr_2^2)\}$.
\end{proof}

\section{Weak Limit Distribution}
In the previous section, we give some sufficient conditions such
that \eqref{eq7} has a unique stationary distribution
$\pi^{n,\3}\in\mathcal {P}(H_n)$ for a fixed $n$ and a sufficiently
small stepsize $\3\in(0,1)$. In this section we proceed to discuss
the weak limit behavior of $\pi^{n,\3}\in\mathcal {P}(H_n)$ and give
positive answers to the following questions:

\begin{itemize}
\item Will the stationary distribution
$\pi^{n,\3}(\cdot)$ converge weakly to some probability measure in
$\mathcal {P}(H)$ whenever $n\to\8$ and $\3\to0$ ?
\item If yes, what is the weak limit probability measure ?
\end{itemize}

Denote  $\{X_x(t)\}_{t\geq0,x\in H}$ by the mild solution of
\eqref{eq1} starting from the point $x$ at time $t=0$, which is a
homogenous Markov process. For any subset
$\Gamma\subset\mathscr{B}(H)$ and arbitrary $t\ge0$,  let
$\mathbb{P}_t(x,\Gamma):=\mathbb{P}(X_x(t)\in\Gamma). $

\begin{defn}
$\{X_x(t)\}_{t\ge0,x\in H}$ is said to have a stationary
distribution $\pi(\cdot)\in\mathcal {P}(H)$ if $
\lim\limits_{t\rightarrow\infty}\d_{\mathbb{L}}(\mathbb{P}_t(x,\cdot),\pi(\cdot))=0.
$
\end{defn}

To reveal the limit behavior of $\pi^{n,\3}(\cdot)$, we first give
several auxiliary lemmas.

\begin{lem}\label{analytic}
{\rm Let (H1)-(H4) hold and assume further that $2\aa+\gg>0$. Then
the mild solution $\{X_x(t)\}_{t\ge0,x\in H}$ of \eqref{eq1} has a
unique stationary distribution $\pi(\cdot)\in\mathcal {P}(H)$. }
\end{lem}

\begin{proof}
We remark that \cite[Theorem 3.1]{BHY} investigates the stationary
distribution of \eqref{eq1} with $\si^0=0$, i.e., the diffusion
coefficient there is a Hilbert-Schmidt operator. For $\si^0\neq0$,
note that $\si$ is not Hilbert-Schmidt. Therefore \cite[Theorem
3.1]{BHY} is unavailable for \eqref{eq1}. Let
\begin{equation}\label{a1}
\bar{Z}(t):=\int_0^te^{(t-s)A}\si^0\d W(s)\mbox{ and }
\bar{X}_x(t):=X_x(t)-\bar{Z}(t).
\end{equation}
Then \eqref{eq1} can be rewritten in the form
\begin{equation}\label{a6}
\d \bar{X}_x(t)=\{A\bar{X}_x(t)+b(X_x(t))\}\d t+\si^1(X_x(t))\d
W(t).
\end{equation}
To be precise, \eqref{a6} is first meant in the mild sense. But
under (H1)-(H3) it also has a unique variation solution, and
therefore the It\^o formula applies to $\|\bar{X}_x(t)\|_H^2$.
Carrying out similar arguments to those of Theorem \ref{boundedness}
and Theorem \ref{stability} respectively, for some  bounded subset
$U\subseteq H$ we deduce that
\begin{equation}\label{eq2}
\sup_{t\ge0}\sup_{x\in U}\E\|X_x(t)\|^2_H<\8
\end{equation}
and
\begin{equation*}
\lim\limits_{t\to\8}\sup_{x,y\in U}\E\|X_{x}(t)-X_{y}(t)\|_H^2=0.
\end{equation*}
Then, following the argument of that of \cite[Theorem 3.1]{BHY}
yields the desired assertion.
\end{proof}

\begin{lem}\label{num}
{\rm Let (H1) and (H2) hold and assume further that there exists
$\dd_2>0$ and $\theta_2\in(0,1)$ such that
\begin{equation}\label{eq38}
\int_0^\3\|e^{sA}\si^0\|_{HS}^2\d s\le \dd_2 \3^{\theta_2}.
\end{equation}
Then
\begin{equation}\label{e5}
\sup_{t\ge0}\mathbb{E}\|\bar{Z}(t)-\bar{Z}(\lfloor
t\rfloor)\|_H^2\le C\3^{\theta_1\wedge\theta_2},
\end{equation}
where $C>0$ is a constant independent of $\3$. }
\end{lem}

\begin{proof}
 Recall from \cite[Theorem 6.13,
p.74]{Pz83} that there exists $C_1>0$ such that
\begin{equation}\label{a9}
\|(-A)^{\aa_1}e^{t A}\|\le C_1t^{-\aa_1},\ \ \
\|(-A)^{-\aa_2}(1-e^{tA})\|\leq C_1t^{\aa_2},
\end{equation}
for arbitrary $\aa_1\ge0,\ \aa_2\in[0,1]$, and that
\begin{equation}\label{a10}
(-A)^{\aa_3+\aa_4}x=(-A)^{\aa_3}(-A)^{\aa_4}x,\ \ \ \ x\in\mathcal
{D}((-A)^{\gamma}),
\end{equation}
for any $\aa_3,\aa_4\in\R$, where
$\gamma:=\max\{\aa_3,\aa_4,\aa_3+\aa_4\}$. In the light of the
independent increment of Wiener process and the It\^o's isometry,
 \begin{equation*}
\begin{split}
\E\|\bar{Z}(t)-\bar{Z}(\lf t\rf)\|_H^2
%&=\int_0^te^{(t-\lf
%s\rf)A}\si^0\d
%W(s)-\int_0^{\lf t\rf}e^{(\lf t\rf-\lf s\rf)A}\si^0\d W(s)\\
%&\le2\E\Big\|\int_0^{\lf t\rf}(e^{(t-\lf t\rf)A}-I)e^{(\lf t\rf-\lf
%s\rf)A}\si^0_n\d W(s)\Big\|_H^2\\
%&\quad+2\E\Big\|\int_{\lf t\rf}^te^{(t-\lf s\rf)A}\si^0_n\d
%W(s)\Big\|_H^2\\
&=\int_0^{\lf t\rf}\|(\e^{(t-\lf t\rf)A}-\mbox{id}_H)\e^{(\lf t\rf-
s)A}\si^0\|_{HS}^2\d s+\int_{\lf t\rf}^t\|\e^{(t-
s)A}\si^0\|_{HS}^2\d s.
\end{split}
\end{equation*}
This, combining (H2), \eqref{eq38},  \eqref{a9} with \eqref{a10},
yields that
 \begin{equation*}
\begin{split}
\E\|\bar{Z}(t)-\bar{Z}(\lf t\rf)\|_H^2&\le\int_0^{\lf
t\rf}\|(-A)^{-\theta_1}(\e^{(t-\lf
t\rf)A}-\mbox{id}_H)\|^2\cdot\|(-A)^{\theta_1}\e^{(\lf
t\rf-s)A}\si^0\|_{HS}^2\d
s\\
&\quad+\int_0^\3\|\e^{sA}\si^0\|_{HS}^2\d s\\
&\le C_1^2\3^{2\theta_1}\int_0^{\lf
t\rf}\|(-A)^{\theta_1}\e^{sA}\si^0\|_{HS}^2\d s+2\dd_2\3^{\theta_2}\\
&\le (C_1^2\dd_1+\dd_2) \3^{\theta_1\wedge\theta_2},
\end{split}
\end{equation*}
and therefore the desired assertion follows.
\end{proof}

\begin{rem}
{\rm Let $\si^0=\mbox{id}_H$ and $A$ be the Laplace operator defined
in Remark \ref{Laplace}. A straightforward computation shows that
\begin{equation}\label{e7}
\int_0^\3\|e^{sA}\|_{HS}^2\d
s=\ff{1}{2}\sum_{k=1}^\8\ff{1}{k^2}(1-e^{-2k^2\3}).
\end{equation}
Recall that for arbitrary $\dd\in(0,1)$ and $x,y\ge0$
\begin{equation}\label{e8}
|e^{-x}-e^{-y}|\le |x-y|^\dd.
\end{equation}
It then follows from \eqref{e7} and \eqref{e8} that
\begin{equation*}
\int_0^\3\|e^{sA}\|_{HS}^2\d s\le 2^{\dd-1}
\3^\dd\sum_{k=1}^\8\ff{1}{k^{2(1-\dd)}}.
\end{equation*}
Hence, \eqref{eq38}  holds with $\dd_2=2^{\dd-1}
\sum_{k=1}^\8\ff{1}{k^{2(1-\dd)}}$ and $\theta_2=\dd\in(0,1/2)$.

}

\end{rem}

\begin{lem}\label{numer}
{\rm Let the assumptions of Lemma \ref{analytic} hold and
\begin{equation}\label{e6}
\tau:=\aa^{-1}L_1+(2\aa)^{-1/2}L_2\in(0,1).
\end{equation}
Then
\begin{equation}\label{eq43}
\sup_{t\ge0}\E\|X_x(t)-Y^{n,\3}_{x_n}(t)\|_H^2\le
C\{\ll_n^{-(\theta_1\wedge1/2)}+\triangle^{\theta_1\wedge\theta_2}\},
\end{equation}
where $C>0$ is a constant dependent on $x\in H$ but independent of
$n$ and $\3$. }
\end{lem}
\begin{proof}
%By the H\"older inequality, the It\^o isometry, (H1) and
%\eqref{eq16}, we deduce  that
%\begin{equation}\label{f4}
%\sup_{0\leq t\leq T}\mathbb{E}\|X(t)\|^2_H\leq C
%\end{equation}
%for some $C_T>0$. This, together with \eqref{eq16}, further yields
%that
By \eqref{eq16} and \eqref{eq2}, it follows that
\begin{equation}\label{f3}
\sup_{t\ge0}\E \|b(X_x(t))\|_H^2+
\sup_{t\ge0}\E\|\sigma^1(X_x(t))\|_{HS}^2\le C.
\end{equation}
Note that $(\E\|\cdot\|_H^2)^{1/2}$ is a norm and recall from
\cite[Theorem 202]{HLP} the Minkowski integral inequality:
\begin{equation*}
\Big(\E\Big|\int_0^tF(s)\d
s\Big|^2\Big)^{1/2}\le\int_0^t\Big(\E|F(s)|^2\Big)^{1/2}\d s,\ \ \
t\ge0,
\end{equation*}
where $F : [0,\8)\times \OO\to R$ is measurable and locally
integrable. Then, applying the It\^o isometry and using (H1), we
obtain from \eqref{eq35} that
\begin{equation}\label{a0}
\begin{split}
&(\mathbb{E}\|\bar{X}_x(t)-\bar{X}_x(\lfloor
t\rfloor)\|_H^2)^{1/2}\\
&\leq \|\e^{\lfloor t\rfloor A}\{\e^{(t-\lfloor
t\rfloor)A}-1\}x\|_H\\
&\quad+\int_0^{\lfloor t\rfloor}(\mathbb{E}\|\e^{(\lfloor t\rfloor-s)A}\{\e^{(t-\lfloor t\rfloor)A}-\mbox{id}_H\}b(X_x(s))\|_H^2)^{1/2}\d s\\
&\quad+\Big(\int_0^{\lfloor t\rfloor}\mathbb{E}\|\e^{(\lfloor t\rfloor-s)A}\{\e^{(t-\lfloor t\rfloor)A}-\mbox{id}_H\}\sigma^1(X_x(s))\|_{HS}^2\d s\Big)^{1/2}\\
&\quad+\int_{\lfloor
t\rfloor}^t(\mathbb{E}\|b(X_x(s))\|_H^2)^{1/2}\d
s+\Big(\int_{\lfloor t\rfloor}^t\E\|\sigma^1(X_x(s))\|_{HS}^2\d
s\Big)^{1/2}
\\&=:F_1(t)+F_2(t)+F_3(t)+F_4(t)+F_5(t).
\end{split}
\end{equation}
Let $\rr:=(\theta_1\wedge\theta_2)/2$. In view of \eqref{a9},
\eqref{a10}, (H1) and the boundedness of $(-A)^{-(1-\rr/2)}$, one
has
\begin{equation*}
\begin{split}
F_1(t)&=\|(-A)^{-(1-\rr/2)}\e^{\lfloor t\rfloor A}(-A)^{-\rr/2}
\{\e^{(t-\lfloor t\rfloor)A}-\mbox{id}_H\}(-A)x\|_H^2\\
&\leq\|(-A)^{-(1-\rr/2)}\e^{\lfloor t\rfloor
A}\|^2\cdot\|(-A)^{-\rr/2}
\{\e^{(t-\lfloor t\rfloor)A}-\mbox{id}_H\}(-A)x\|_H^2\\
&\leq C\|(-A)^{-(1-\rr/2)}\|^2\cdot\|Ax\|_H^2\triangle^\rr.
\end{split}
\end{equation*}
Also, by \eqref{a9} and \eqref{a10}, we obtain from   \eqref{f3}
that for $\tilde{\theta}\in(0,1)$
\begin{equation}\label{f7}
\begin{split}
&\sum_{k=2}^5F_k(t)\\
% &\leq \int_0^{\lfloor t\rfloor}\|e^{(\lfloor
%t\rfloor-s)A}\{e^{(t-\lfloor
%t\rfloor)A}-1\}\|(\E \|b(X_x(s))\|_H^2)^{1/2}\d s\\
%&\quad+\Big(\int_0^{\lfloor t\rfloor}\|e^{(\lfloor
%t\rfloor-s)A}\{e^{(t-\lfloor
%t\rfloor)A}-\mbox{id}_H\}\|^2\mathbb{E}\|\sigma^1(X_x(s))\|_{HS}^2\d s\Big)^{1/2}\\
%&\quad+\int_{\lfloor t\rfloor}^t(\mathbb{E}\|b(X_x(s))\|_H^2)^{1/2}\d
%s+\Big(\int_{\lfloor t\rfloor}^t\E\|\sigma^1(X_x(s))\|_{HS}^2\d
%s\Big)^{1/2}\\
&\leq C\3^{1/2}+C\int_0^{\lfloor
t\rfloor}\|(-A)^{\rr}\e^{\tilde{\theta}(\lfloor
t\rfloor-s)A}\|\cdot\|\e^{(1-\tilde{\theta})(\lfloor
t\rfloor-s)A}\|\cdot\|(-A)^{-\rr}\{\e^{(t-\lfloor
t\rfloor)A}-1\}\|\d s\\
&\quad+C\Big(\int_0^{\lfloor
t\rfloor}\|(-A)^{\rr}\e^{\tilde{\theta}(\lfloor
t\rfloor-s)A}\|^2\cdot\|\e^{(1-\tilde{\theta})(\lfloor
t\rfloor-s)A}\|^2\cdot\|(-A)^{-\rr}\{\e^{(t-\lfloor
t\rfloor)A}-1\}\|^2\d
s\Big)^{1/2}\\
%&\leq C\3^{1/2}+C\3^\rr\int_0^{\lfloor
%t\rfloor}\|(-A)^{\rr}e^{\tilde{\theta} sA}\|\cdot\|e^{(1-\tilde{\theta})sA}\|\d s\\
%&\quad+C\3^\rr\Big(\int_0^{\lfloor
%t\rfloor}\|(-A)^{\rr}e^{\tilde{\theta}sA}\|^2\cdot\|e^{(1-\tilde{\theta})sA}\|^2\d
%s\Big)^{1/2}\\
&\leq C\3^{1/2}+C\3^{\rr}\int_0^{\lfloor t\rfloor}(\tilde{\theta}
s)^{-\rr}\e^{-\aa(1-\tilde{\theta})s}\d
s+C\3^{\rr}\Big(\int_0^{\lfloor t\rfloor}(\tilde{\theta}
s)^{-2\rr}\e^{-2\aa(1-\tilde{\theta})s}\d s\Big)^{1/2}.\\
\end{split}
\end{equation}
Observe that
\begin{equation*}
\int_0^{\lfloor t\rfloor} s^{-\rr}\e^{-\aa(1-\tilde{\theta})s}\d
s\le(\aa(1-\tilde{\theta}))^{\rr-1}\int_0^\8s^{-\rr}\e^{-s}\d
s=(\aa(1-\tilde{\theta}))^{\rr-1}\Gamma(1-\rr),
\end{equation*}
and similarly
\begin{equation*}
\int_0^{\lfloor t\rfloor} s^{-2\rr}e^{-2\aa(1-\tilde{\theta})s}\d
s\le(2\aa(1-\tilde{\theta}))^{2\rr-1}\Gamma(1-2\rr),
\end{equation*}
where $\GG(\cdot)$ is the Gamma function. Hence
\begin{equation*}
\sum_{k=2}^4F_k(t)\le C\3^{(\theta_1\wedge\theta_2)/2}.
\end{equation*}
This, together with the estimate of $F_1(t)$, gives that
\begin{equation*}
\sup_{t\ge0}\mathbb{E}\|\bar{X}_x(t)-\bar{X}_x(\lfloor
t\rfloor)\|_H^2\le C\3^{\theta_1\wedge\theta_2}.
\end{equation*}
Noting that $\bar{X}_x(t)=X_x(t)-\bar{Z}(t)$ and utilizing
\eqref{e5}, one has
\begin{equation}\label{a4}
\sup_{t\ge0}\mathbb{E}\|X_x(t)-X_x(\lfloor t\rfloor)\|_H^2\le
C\3^{\theta_1\wedge\theta_2}.
\end{equation}
%\begin{equation*}
%\begin{split}
%\pi_ne^{tA}x&=\sum_{j=1}^n\<e^{tA}x,e_j\>_He_j\\
%&=\sum_{j=1}^n\sum_{k=1}^\8\<x,e_k\>_H\<e^{tA}e_k,e_j\>_He_j\\
%&=\sum_{j=1}^n\<x,e_j\>_H\<e^{tA}e_j,e_j\>_He_j\\
%&=\sum_{j=1}^n\<x,e_j\>_He^{-\ll_jt}e_j\\
%&=\sum_{j=1}^n\<x,e_j\>_He^{tA}e_j\\
%&=e^{tA}\Big(\sum_{j=1}^n\<x,e_j\>_He_j\Big)\\
%&=e^{tA}\pi_nx\\
%\end{split}
%\end{equation*}
Since
\begin{equation*}
\begin{split}
\|(\mbox{id}_H-\pi_n)(-A)^{-\theta_1}u\|^2_H
%&=\Big\|(\mbox{id}_H-\pi_n)(-A)^{-\theta_1}\Big(\sum_{k=1}^\8\<x,e_k\>_He_k\Big)\Big\|_H^2\\
%&=\Big\|\sum_{k=1}^\8\ll_k^{-\theta_1}\<x,e_k\>_H(\mbox{id}_H-\pi_n)e_k\Big\|_H^2\\
&=\Big\|\sum_{k=n+1}^\8\ll_k^{-\theta_1}\<u,e_k\>_He_k\Big\|_H^2\\
&\le\ll_n^{-2\theta_1}\|u\|_H^2,\ \ \ u\in H,
\end{split}
\end{equation*}
we arrive at
\begin{equation}\label{f6}
\|(\mbox{id}_H-\pi_n)(-A)^{-\theta_1}\|^2\le\ll_n^{-2\theta_1}.
\end{equation}
 By virtue of  the It\^o isometry, (H2), \eqref{f6}, \eqref{a9} and
\eqref{a10}, it follows that
\begin{equation}\label{a7}
\begin{split}
\E\|\bar{Z}(t)-Z^{n,\3}(t)\|_H^2&\le2\int_0^t\|\e^{sA}(\mbox{id}_H-\pi_n)\si^0\|_{HS}^2\d
s\\
&\quad+2\int_0^t\|(-A)^{-\theta_1}(\mbox{id}_H-\e^{(s-\lf s\rf)}A)(-A)^{\theta_1}\e^{(t-s)A}\si^0_n\|_{HS}^2\d s\\
%&=\sum_{k=1}^\8\int_0^t\|e^{sA}(\mbox{id}_H-\pi_n)\si^0e_k\|_H^2\d
%s\\
%&=\sum_{k=1}^\8\int_0^t\|(\mbox{id}_H-\pi_n)e^{sA}\si^0e_k\|_H^2\d
%s\\
%&=\int_0^t\|(\mbox{id}_H-\pi_n)e^{sA}\si^0\|_{HS}^2\d
%s\\
&\le
2\|(\mbox{id}_H-\pi_n)(-A)^{-\theta_1}\|^2\int_0^t\|(-A)^{\theta_1}\e^{sA}\si^0\|_{HS}^2\d
s\\
&\quad+ C\3^{2\theta_1}\int_0^t\|(-A)^{\theta_1}\e^{sA}\si^0_n\|_{HS}^2\d s\\
%&\le \dd_1\|(\mbox{id}_H-\pi_n)(-A)^{-2\theta_1}\|^2\\
&\le
C(\|(\mbox{id}_H-\pi_n)(-A)^{-\theta_1}\|^2+\3^{2\theta_1})\int_0^t\|(-A)^{\theta_1}\e^{sA}\si^0\|_{HS}^2\d
s\\
&\le C(\ll_n^{-2\theta_1}+\3^{2\theta_1}).
\end{split}
\end{equation}
%where we have also used $\pi_ne^{tA}x=e^{tA}\pi_nx, x\in H$.
Following the argument of \eqref{a0}, we have
\begin{equation}\label{a5}
\begin{split}
&(\E\|\bar{X}_x(t)-\tt{Y}^{n,\3}_{x_n}(t)\|_H^2)^{1/2}\\&\le \|e^{tA}(\mbox{id}_H-\pi_n)x\|_H\\
&\quad+\int_0^t\|\e^{(t-s)A}(\mbox{id}_H-\pi_n)\|(\E\|b(X_x(s))\|_H^2)^{1/2}\d s\\
&\quad+\Big(\int_0^t\|\e^{(t-s)A}(\mbox{id}_H-\pi_n)\|^2\E\|\si^1(X_x(s))\|_{HS}^2\d s\Big)^{1/2}\\
&\quad+\int_0^t\|\e^{(t-s)A}\|(\E\|b_n(X_x(s))-b_n(X_x(\lf s\rf))\|_H^2)^{1/2}\d s\\
&\quad+\Big(\int_0^t\|\e^{(t-s)A}\|^2\E\|\si_n^1(X_x(s))-\si_n^1(X_x(\lf
s\rf))\|_{HS}^2\d s\Big)^{1/2}\\
&\quad+\int_0^t\|\e^{(t-s)A}\|(\E\|b_n(X_x(\lf
s\rf))-b_n(Y^{n,\3}_{x_n}(\lf
s\rf))\|_H^2)^{1/2}\d s\\
&\quad+\Big(\int_0^t\|\e^{(t-s)A}\|^2\E\|\si_n^1(X_x(\lf
s\rf))-\si_n^1(Y^{n,\3}_{x_n}(\lf
s\rf))\|_{HS}^2\d s\Big)^{1/2}\\
&\quad+\int_0^t\|\e^{(t-s)A}\{\mbox{id}_H-\e^{(s-\lf s\rf)
A}\}\|(\E\|b(Y^{n,\3}_{x_n}(\lf
s\rf))\|^2_H)^{1/2}\d s\\
&\quad+\Big(\int_0^t\|\e^{(t-s)A}\{\mbox{id}_H-\e^{(s-\lf s\rf)
A}\}\|^2\E\|\si^1(Y^{n,\3}_{x_n}(\lf
s\rf))\|^2_{HS}\d s\Big)^{1/2}\\
 &=:\sum_{i=1}^9G_i(t).
\end{split}
\end{equation}
%\begin{equation}\label{a5}
%\begin{split}
%\bar{X}_x(t)-\bar{Z}^n(t)&=e^{tA}x-e^{tA}\pi_nx\\
%&\quad+\int_0^te^{(t-s)A}(\mbox{id}_H-\pi_n)\{b(X_x(s))\d s+\si^1(X_x(s))\d W(s)\}\\
%&\quad+\int_0^te^{(t-s)A}\{(b_n(X_x(s))-b_n(X(\lf s\rf)))\d s\\
%&\quad+(\si_n^1(X_x(s))-\si_n^1(X(\lf s\rf)))\d W(s)\}\\
%&\quad+\int_0^te^{(t-s)A}\{(b_n(X(\lf s\rf))-b_n(Y^n(\lf s\rf)))\d s\\
%&\quad+(\si_n^1(X(\lf s\rf))-\si_n^1(Y^n(\lf s\rf)))\d W(s)\}\\
%&\quad+\int_0^te^{(t-s)A}\{1-e^{(s-\lf s\rf) A}\}\{b_n(Y^n(\lf
%s\rf))\d s+\si_n^1(Y^n(\lf
%s\rf))\d W(s)\}\\
%&=:\bar{I}_1(t)+\bar{I}_2(t)+\bar{I}_3(t)+\bar{I}_4(t)+\bar{I}_5(t).
%\end{split}
%\end{equation}
A straightforward computation shows that
\begin{equation*}
\|\e^{tA}(\mbox{id}_H-\pi_n)u\|_H^2=\sum_{i=n+1}^\8\e^{-2\ll_it}\<u,e_i\>_H^2,\
\ \ u\in H.
\end{equation*}
This further gives that
\begin{equation}\label{f2}
\|\e^{tA}(\mbox{id}_H-\pi_n)\|^2\le \e^{-2\ll_nt}
\end{equation}
and that
\begin{equation}\label{f8}
G_1(t)\le\Big(\sum_{i=n+1}^\8\ff{\e^{-2\ll_it}}{\ll_i^2}\ll_i^2\<x,e_i\>_H^2\Big)^{1/2}\le
\ll_n^{-1}\|Ax\|_H
\end{equation}
by recalling that $\{\ll_i\}_{i\ge1}$ is a nondecreasing sequence.
By \eqref{f3} and \eqref{f2}, one has
\begin{equation}\label{e1}
\begin{split}
G_2(t)+G_3(t)&\le C\int_0^t\|\e^{(t-s)A}(\mbox{id}_H-\pi_n)\|\d
s+C\Big(\int_0^t\|\e^{(t-s)A}(\mbox{id}_H-\pi_n)\|^2\d s\Big)^{1/2}\\
&\le C\int_0^t\e^{-\ll_n(t-s)}\d s+C\Big(\int_0^t\e^{-2\ll_n(t-s)}\d
s\Big)^{1/2}\le C(\ll_n^{-1}+\ll_n^{-1/2}).
\end{split}
\end{equation}
Taking (H1), (H3) and \eqref{a4} into account gives that
\begin{equation}\label{e2}
\begin{split}
G_4(t)+G_5(t)&\le
C\3^{(\theta_1\wedge\theta_2)/2}\Big\{\int_0^t\|\e^{(t-s)A}\|\d
s+\Big(\int_0^t\|\e^{(t-s)A}\|^2\d s\Big)^{1/2}\Big\}\\
&\le C\3^{(\theta_1\wedge\theta_2)/2}.
\end{split}
\end{equation}
Next, note from (H1) and (H3) that
\begin{equation}\label{e3}
\begin{split}
&G_6(t)+G_7(t)\\
%&\le\int_0^t\|e^{(t-s)A}\|(\E\|b(X(\lf s\rf))-b(Y^n(\lf
%s\rf))\|_H^2)^{1/2}\d s\\
%&\quad+\Big(\int_0^t\|e^{(t-s)A}\|^2\E\|\si^1(X(\lf
%s\rf))-\si^1(Y^n(\lf
%s\rf))\|_{HS}^2\d s\Big)^{1/2}\\
&\le\sup_{0\le s\le t}(\E\|b(X_x(\lf s\rf))-b(Y^{n,\3}_{x_n}(\lf
s\rf))\|_H^2)^{1/2}\int_0^t\|\e^{(t-s)A}\|\d s\\
&\quad+\sup_{0\le s\le t}(\E\|\si^1(X_x(\lf
s\rf))-\si^1(Y^{n,\3}_{x_n}(\lf
s\rf))\|_H^2)^{1/2}\Big(\int_0^t\|\e^{(t-s)A}\|^2\d s\Big)^{1/2}\\
&\le\aa^{-1}\sup_{0\le s\le t}(\E\|b(X_x(\lf
s\rf))-b(Y^{n,\3}_{x_n}(\lf
s\rf))\|_H^2)^{1/2}\\
&\quad+(2\aa)^{-1/2}\sup_{0\le s\le t}(\E\|\si^1(X_x(\lf
s\rf))-\si^1(Y^{n,\3}_{x_n}(\lf
s\rf))\|_H^2)^{1/2}\\
&\le\tau\sup_{0\le s\le t}(\E\|X_x(s)-Y^{n,\3}_{x_n}(
s)\|_H^2)^{1/2}\\
&\le\tau\sup_{0\le s\le
t}(\E\|\bar{X}_x(s)-\tt{Y}^{n,\3}_{x_n}(s)\|_H^2)^{1/2}+\tau\sup_{0\le
s\le t}(\E\|\bar{Z}(s)-Z^{n,\3}(s)\|_H^2)^{1/2},
\end{split}
\end{equation}
where $\tau\in(0,1)$ is defined by \eqref{e6}. Following the
argument of \eqref{f7} leads to
\begin{equation}\label{e4}
G_8(t)+G_9(t)\le C\3^{(\theta_1\wedge\theta_2)/2}.
\end{equation}
Substituting \eqref{f8}-\eqref{e4} into \eqref{a5} yields that
\begin{equation*}
\begin{split}
\sup_{t\ge0}(\E\|\bar{X}_x(t)-\tt{Y}^{n,\3}_{x_n}(t)\|_H^2)^{1/2}&\le
C(\ll_n^{-1/2}+\3^{(\theta_1\wedge\theta_2)/2})
\end{split}
\end{equation*}
due to $\tau\in(0,1)$. Consequently the desired assertion follows
from \eqref{a7}.
\end{proof}

\begin{thm}
{\rm Assume that (H1)-(H4), \eqref{eq38} and \eqref{e6} hold. Then,
there exists a $\3_n$ such that $\lim_{n\to\8}\3_n=0$
\begin{equation*}
\lim_{n\to\8}\d_{\mathbb{L}}(\pi^{n,\3_n}(\cdot),\pi(\cdot))=0.
\end{equation*}
}
\end{thm}

\begin{proof}
Fix $x\in H$ and let $\ee>0$ be arbitrary.
 By Lemma \ref{numer}, there exist a sufficiently large
 $n\in\mathbb{N}$ and a $\bar{\3}_n$ sufficiently small such that
such that
\begin{equation*}
\begin{split}
\d_{\mathbb{L}}(\mathbb{P}_{k\bar{\3}_n}(x,\cdot),\mathbb{P}^{n,\bar{\3}_n}_{k}(x_{n},\cdot))\le\ee/3.
\end{split}
\end{equation*}
For the previous  $n\in\mathbb{N}$, by Theorem \ref{numerical},
there exist a sufficiently small $\tt{\3}_n$ and $T_1>0$ such that
\begin{equation*}
\d_{\mathbb{L}}(\mathbb{P}^{n,\tt{\3}_n}_k(x_n,\cdot),\pi^{n,\tt{\3}_n}(\cdot))\le
\ee/3
\end{equation*}
whenever $k\tt{\3}_n\ge T_1$. Furthermore, due to Lemma
\ref{analytic} there exists  $T_2>0$ such that
\begin{equation*}
\d_{\mathbb{L}}(\mathbb{P}_{t}(x,\cdot),\pi(\cdot))\le \ee,\ \ \
t\ge T_2.
\end{equation*}
Let $T:=T_1\vee T_2$, $\3_n=\bar{\3}_n\wedge\tt{\3}_n$ and
$k=[T/\3_n]+1$. Then the desired assertion follows from
 the
triangle inequality
\begin{equation*}
\begin{split}
\d_{\mathbb{L}}(\pi^{n,\3}(\cdot),\pi(\cdot))&\leq
\d_{\mathbb{L}}(\mathbb{P}_{k\3}(x,\cdot),\pi(\cdot))+d_{\mathbb{L}}(\mathbb{P}_{k\3}(x,\cdot),
\mathbb{P}^{n,\3}_k(x_n,\cdot))\\
&\quad+\d_{\mathbb{L}}(\mathbb{P}^{n,\3}_k(x_n,\cdot),\pi^{n,\3}(\cdot)).
\end{split}
\end{equation*}
\end{proof}

\begin{rem}
{\rm For the finite-dimensional case, finite-time convergence of
numerical scheme is enough to discuss the limit of stationary
distribution of numerical solution  \cite[Theorem 6.23, p.266]{my06}.
While for the infinite-dimensional case, we need the  {\it uniform
convergence} of EI scheme \eqref{eq7} to reveal  the limit behavior
of $\pi^{n,\3}$, which is quite different from the
finite-dimensional cases,  and therefore \eqref{e6} is imposed. On
the other hand, for the finite-time convergence of EM scheme
\eqref{eq6}, condition \eqref{e6}  can be deleted by checking the
argument of Lemma \ref{numer} and combining with the Gronwall
inequality. }
\end{rem}

\begin{rem}
{\rm By following the procedure of this paper, numerical
approximation of stationary distribution of SPDEs with jumps can
also be discussed, which will be reported in forthcoming paper. }
\end{rem}

\noindent{\bf{Acknowledge }} The authors wish to express their
sincere thanks to the anonymous referee for his/her careful comments
and valuable suggestions, which greatly improved the paper.

 \end{document}